\documentclass[11pt, a4paper, reqno]{amsart}

\usepackage{amsmath}
\usepackage{amstext}
\usepackage{amssymb}
\usepackage{amsthm}
\usepackage{enumerate}
\usepackage{enumitem}
\usepackage{mathrsfs}
\usepackage[all]{xy}
\usepackage{datetime}

\usepackage[svgnames]{xcolor}
\usepackage{tikz}
\usepackage{enumerate}

\pgfdeclarelayer{edgelayer}
\pgfdeclarelayer{nodelayer}
\pgfsetlayers{edgelayer,nodelayer,main}

\tikzstyle{none}=[inner sep=0pt]

\usepackage[pdfauthor={Asger Tornquist, David Schrittesser},
            pdftitle={Definable maximal discrete sets in forcing extensions},
            pdfproducer={Latex with hyperref}]{hyperref}

\makeatletter
\def\imod#1{\allowbreak\mkern10mu({\operator@font mod}\,\,#1)}
\makeatother

\newtheorem{theorem}{Theorem}[section]

\newtheorem{prop}[theorem]{Proposition}
\newtheorem{lemma}[theorem]{Lemma}
\newtheorem{claim}[theorem]{Claim}
\newtheorem{corollary}[theorem]{Corollary}

\theoremstyle{definition}
\newtheorem{definition}[theorem]{Definition}

\newtheorem{fact}[theorem]{Fact}

\theoremstyle{remark}
\newtheorem{remark}[theorem]{Remark}

\theoremstyle{remark}

\numberwithin{equation}{section}

    \DeclareMathOperator{\lh}{lh}

    \DeclareMathOperator{\dom}{dom}

    \DeclareMathOperator{\grph}{graph}

    \newcommand{\upharpoonrighttrict}{\!\upharpoonright\!}

    \newcommand{\forces}{\Vdash}

    \newcommand{\sm}{\ \!}

\def\bP{{\mathbb P}}

\def\bS{{\mathbb S}}
\def\bM{{\mathbb M}}

\newcommand{\baseset}{\Omega}

\renewcommand{\restriction}{\mathord{\upharpoonright}}

\title{Definable maximal discrete sets in forcing extensions}

\author{David Schrittesser}
\author{Asger T\"ornquist}

\address{Department of Mathematical Sciences, University of Copenhagen, Universitetsparken 5, 2100 Copenhagen, Denmark}

\email{asgert@math.ku.dk}

\email{david@logic.univie.ac.at}

\curraddr{(David Schrittesser) University of Toronto, Ontaria, Canada; Harbin Institute of Technology, Harbin, China}


\date{\today}


\begin{document}

\begin{abstract}
Let $\mathcal R$ be a $\Sigma^1_1$ binary relation, and recall that a set $A$ is $\mathcal R$-discrete if no two elements of $A$ are related by $\mathcal R$. We show that in the Sacks and Miller forcing extensions of $L$ there is a $\Delta^1_2$ maximal $\mathcal{R}$-discrete set. We use this to answer in the negative the main question posed in \cite{Fischer2010} by showing that in the Sacks and Miller extensions there is a $\Pi^1_1$ maximal orthogonal family (``mof'') of Borel probability measures on Cantor space. By contrast, we show that if there is a Mathias real over $L$ then there are no $\Sigma^1_2$ mofs.
\end{abstract}

\maketitle

\section{Introduction}\label{s.intro}

{\bf (A)} The present paper studies the definability of what in various contexts is called either \emph{independent sets}, \emph{orthogonal families}, or \emph{antichains}. To capture these notions at once, we adopt the nomenclature of \cite{miller} and make the following definition:

\begin{definition} Let $\mathcal R$ be a binary relation on a set $X$. A set $A\subseteq X$ is called \emph{$\mathcal R$-discrete} iff
\[
(\forall x,y\in A)\ x\neq y\implies \neg ( x \mathbin{\mathcal{R}} y).
\]
By a \emph{maximal $\mathcal R$-discrete set} we mean an $\mathcal R$-discrete set which is maximal under inclusion among $\mathcal R$-discrete sets.
\end{definition}
The above definition is familiar in the context of graphs, i.e., symmetric irreflexive relations, where discrete sets are often also called \emph{independent sets}. 

Another situation in which maximal discrete sets are of interest is when $\mathcal R$ is a \emph{compatibility relation}, i.e., when $\mathcal R$ is symmetric and reflexive.
Such relations often arise from a preorder: if $\preceq$ is a preorder, then the associated compatibility relation $\mathcal R_{\preceq}$ is defined by
$$
x\ \mathcal R_{\preceq}\  y\iff (\exists z) z\preceq x\wedge z\preceq y.
$$
In this context, an $\mathcal R_{\preceq}$-discrete set is often called an \emph{antichain} for $\preceq$.

A straight-forward transfinite induction shows that maximal discrete sets always exist for any binary relation $\mathcal R$. However, the \emph{definability} of such maximal discrete sets may be contentious. In G\"odel's constructible universe $L$ any $\Sigma^1_1$ (i.e., \emph{effectively} analytic) binary relation admits a $\Delta^1_2$ maximal discrete set, a fact that follows routinely from the existence of a $\Delta^1_2$ wellordering of the reals in $L$ of order type $\omega_1$ with a good coding of initial segments. On the other hand, if we let $\Gamma$ be the $F_\sigma$ (in fact $\Sigma^0_2$) graph on $2^\omega$ where $x\ \! \Gamma\ \! y$ iff $x$ and $y$ differ on exactly one bit, then a routine Baire category argument shows that $\Gamma$ admits no Baire measurable maximal discrete set, and so by \cite[Theorem~0.10]{ihoda1989delta} there is no $\Delta^1_2$ maximal $\Gamma$-discrete set if there is a Cohen real over $L$. The situation is parallel with random reals.

The first goal of this paper is to show that the above failure does not persist in all forcing extensions of $L$ with new reals.

\begin{theorem}\label{t.discrete}
Let $\mathcal R$ be a $\Sigma^1_1$ binary relation on an effectively presented Polish space, and let $x$ be a Sacks or Miller real over $L$. Then there is a $\Delta^1_2$ maximal $\mathcal R$-discrete set in $L[x]$.
\end{theorem}

A suitably relativized version of the previous theorem applies more generally to $\mathcal R$ which are $\Sigma^1_1[a]$ for some real parameter $a$.

\medskip

{\bf (B)} Our main application of Theorem \ref{t.discrete} is to the compatibility relation that comes from \emph{absolute continuity} of Borel probability measures. Recall that if $\mu$ and $\nu$ are (non-trivial) measures on a measurable space $X$, then we write $\mu\ll\nu$ just in case every set which is null for $\nu$ is also null for $\mu$. 
Two measures $\mu$ and $\nu$ that are not compatible in $\ll$ (i.e., there is no measure $\ll$-below both of them) 
are called \emph{orthogonal}, written $\mu\perp\nu$. 
By the Lebesgue decomposition theorem, for Borel probability measures this is equivalent to that there exists a Borel set $A\subseteq X$ such that $\nu(A)=1$ and $\mu(A)=0$. 
Also recall that $\mu \sim \nu$ means $\mu\ll\nu \wedge \nu\ll\mu$, i.e., $\mu$ and $\nu$ have the same null sets.

Orthogonal families of measures in the Polish space $P(X)$ of Borel probability measures on a Polish space $X$ (see \cite[Theorem~17.23,~p.127]{kechris1995}) show up in many different contexts, including representation theory, ergodic theory, and operator algebras, see e.g.\ \cite{Shlyakhtenko2004, tornquist2009}. Interest in the definability of maximal orthogonal families of measures (abbreviated \emph{mof}s) can be traced back to the following question posed by Mauldin: If $X$ is a perfect Polish space, is there an \emph{analytic} maximal orthogonal family in $P(X)$? The answer turns out to be no, as shown by Preiss and Rataj \cite{preiss-rataj1985}. A new proof of this fact was provided by Kechris and Sofronidis \cite{kechris-sofronidis2001} based on Hjorth's turbulence theory. Later, it was shown by Fischer and the second author \cite{Fischer2010} that if $V=L$ then there is a $\Pi^1_1$ (lightface) mof in $P(2^\omega)$. On the other hand, \cite{Fischer2010} and \cite{Fischer2012} established that if there is a Cohen or random real over $L$ there are no $\Pi^1_1$ mofs. 

The seemingly restrictive nature of $\Pi^1_1$ mofs motivated the following question in \cite{Fischer2010}: \emph{If there is a $\Pi^1_1$ mof in $P(2^\omega)$, must all reals be constructible?} Further compounding the intrigue, we will see in \S \ref{s.indestruct} below that 
no mof remains maximal after adding a real, eliminating any hope of constructing `indestructible' mofs.
Nevertheless, in this paper we will answer the above question in the negative by showing:

\begin{theorem}
If $x$ is a Sacks or Miller real over $L$, then in $L[x]$ there is a $\Pi^1_1$ mof in $P(2^\omega)$.
\end{theorem}

A counterpoint to this is obtained in the last section, where the following is shown:
\begin{theorem}
There are no $\Pi^1_1$ mofs in the Mathias extension of $L$.
\end{theorem}

As another application of Theorem \ref{t.discrete} we obtain a different proof of the following known result concerning maximal almost disjoint families (``mad'' families) of subsets of $\omega$. Such families are precisely the discrete sets for the compatibility relation that arises from the preorder $\subseteq^*$, inclusion modulo a finite set, in $[\omega]^\omega$. The study of the definability of mad families has a long history, see e.g.\ \cite{mathias1977happy, friedman2010projective, fischer2011projective, fischer2013co, tornquist2015}. 
From Theorem \ref{t.discrete} and \cite{tornquist2013} we get:

\begin{fact}\label{f.mad}
If $x$ is a Sacks or Miller real over $L$, then in $L[x]$ there is a $\Pi^1_1$ infinite mad family in $[\omega]^\omega$.
\end{fact}

We want to thank the anonymous referee for pointing out that a stronger result is already established by well-known methods:
In any model of CH and for any reasonable forcing notion $\bP$ which does not add dominating reals one can construct an infinite $\bP$-\emph{indestructible} mad family (see, e.g., \cite{hrusak,brendle}). In $L$, one can check that one can further demand that the indestructible family be $\Pi^1_1$.
Fact~\ref{f.mad} follows, letting $\bP$ be Sacks or Miller forcing.

\medskip

{\bf (C)} The paper is organized as follows: In \S \ref{s.indesctructible} we show there can be no indestructible $\Sigma^1_2[a]$ mof in $L[a]$.
In \S \ref{s.delta^1_2}, after we review some well-known facts about Sacks and Miller forcing,
we prove a slightly more general version of Theorem \ref{t.discrete}.
We also list some general properties of forcings which allow the proof to go through.
In \S \ref{s.mof}, we apply this to mofs and show that if there is a $\Sigma^1_2$ mof, there is a $\Pi^1_1$ mof.
\S \ref{s.mathias} presents an argument that a Mathias real over $L[a]$ rules out the existence of a $\Sigma^1_2[a]$ mof. 
Using the same ideas, we sketch a new proof that there is no analytic mof.
We close in \S \ref{s.questions} with open questions.

\medskip

{\it Acknowledgements}.
We would like to thank Stevo Todor\v{c}evi\'c for making us aware of the analogue of Galvin's theorem for Miller forcing, which we use in the proof of Corollary \ref{c.Galvin}. Further, we would like to thank the anonymous referee for their diligent corrections and many insightful suggestions.

The authors gratefully acknowledge the generous support from Sapere Aude grant no.\ 10-082689/FNU from Denmark's Natural Sciences Research Council.

\section{There is no indestructible mof}\label{s.indesctructible} \label{s.indestruct}

In this brief section we prove that there is no hope of finding a mof that survives in an outer model which has new reals; in particular, there is no mof which is indestructible by any forcing adding a real. 

A key property of $\ll$ that we use is the so-called \emph{ccc-below} property: If $\mu\in P(X)$, then any orthogonal family of measures $\mathcal F$ such that $\nu\ll\mu$ for all $\nu\in\mathcal F$ must be countable.
In fact, the argument only uses the ccc-below property and the simple definability of $\ll$, and that there is a perfect pairwise incompatible set;
so it holds for a wide class of compatibility relations.
We would like to thank the referee for suggesting the present form of the following theorem and its proof (these were also independently suggested by Benjamin D.\ Miller).

\begin{theorem}
Let $V \subseteq W$ be transitive models of (enough of) ZFC. Assume that $\mathcal A \in V$ is an orthogonal family in $P(2^\omega)$ and $W$ contains a new real, i.e., $(2^\omega)^W\setminus (2^\omega)^V \neq \emptyset$. Then $\mathcal A$ is not maximal in $W$.
In particular, if there is a mof $\mathcal A$ such that $\mathcal A \in L$, then $\operatorname{\mathcal{P}}(\omega) \subseteq L$.
\end{theorem}
\begin{proof}
For the following, we view $P(2^\omega)$ as an effectively presented Polish space in precisely the manner described in \cite{Fischer2010}.
From the product measure construction in \cite[p.~1463]{kechris-sofronidis2001} it follows easily that there is a Borel (in fact, $\Pi^0_1$) Cantor subset of $P(2^\omega)$ of pairwise orthogonal measures; let $Y \in V$ be such. Consider the predicate
$$
R = \{ (\mu, \nu )\in Y\times\mathcal A: \mu\not\perp\nu\}.
$$
The ccc-below property of $\ll$ implies that $R^{-1}(\nu) = \{\mu \in Y : \mu \not\perp\nu\}$ is countable for each $\nu \in P(2^\omega)$.
Letting $\nu \in V$ and $\langle \mu_n : n\in \omega\rangle$ be a sequence in $V$ which enumerates $R^{-1}(\nu)^V$,
observe that also
\[
W \vDash R^{-1}(\nu) = \{\mu_n : n\in \omega\}
\]
as the formula on the right is $\Pi^1_2$ (in fact $\Pi^1_1$, since $\perp$ is arithmetical) in a real coding $\langle \mu_n : n\in \omega\rangle$ and $\nu$.
In particular, $R^{-1}(\nu)^W \subseteq V$ for each $\nu \in V$.
But since $(2^\omega)^W\setminus (2^\omega)^V \neq \emptyset$, $W$ has a new branch $\nu'$ through $Y$ and by the previous, $\nu' \perp \mu$ for every $\mu \in \mathcal A$.
\end{proof}

\section{Definable maximal discrete sets in the Sacks or Miller extension}\label{s.delta^1_2}

In this section, we prove Theorem \ref{t.discrete} in the following, slightly stronger form:
\begin{theorem}\label{t.sacks}
Let  $a\in\omega^\omega$.
For any $\Sigma^1_1[a]$ binary relation $\mathcal R$ on $\omega^\omega$ there is a $\Delta^1_2[a]$ predicate which defines a maximal $\mathcal R$-discrete set in both $L[a]$ and $L[a][x]$, where $x$ is any Sacks or Miller real over $L[a]$.
\end{theorem}
The theorem applies to arbitrary effectively presented Polish spaces, since any two uncountable such spaces are $\Delta^1_1$ isomorphic.
The argument also applies more generally to \emph{arboreal} forcing notions satisfying certain conditions, which we list in Theorem \ref{t.arboreal}. 

\medskip

Following the anonymous referee's suggestion, we mention the following version of Theorem~\ref{t.sacks}, in which we allow an arbitrary model of CH (rather than a model of the form $L[a]$) as the ground model. 
The proof is exactly the same, minus the definability considerations. 
The reader can find similar results, e.g., in \cite{brendle-perfect}.
\begin{theorem}\label{t.sacks.ch}
Assume CH holds in $V$, let $x$ be a Sacks or Miller real over $V$,
and let $\mathcal R$ be a $\mathbf{\Sigma}^1_1$ binary relation on $\omega^\omega$. Then there is a sequence 
$\langle T_\alpha : \alpha < \omega_1\rangle \in V$ 
of trees on $\omega$ (i.e, of codes of closed subsets of $\omega^\omega$) such that $\bigcup_{\alpha < \omega_1} [T_\alpha]$ defines a maximal $\mathcal R$-discrete set 
in both $V$ and $V[x]$.
\end{theorem}

\medskip

Before delving into the proof, we collect a few preliminaries about Sacks forcing $\bS$ and Miller forcing $\bM$.
Firstly, we need the following elementary fact about descriptive complexity calculations and the forcing relation. 
Let $\bP \in \{ \bS, \bM \}$.
In either case, we denote by $\dot x_G$ the name for the generic real.

\begin{fact}\label{f.complexityforces}
If $\varphi(x,y)$ is a $\Pi^1_1$ formula, the set 
\[
\{(p,a)\in\bP\times\omega^\omega:p\forces_{\bP} \varphi(\dot x_G,\check a)\}
\]
is $\Pi^1_1$.
\end{fact}
\begin{proof}
We treat the case $\bS$ in detail. 
Clearly,  $p\forces_{\bS} \varphi(\dot x_G,\check a)$ if and only if the analytic set 
\begin{equation}\label{set}
A= \{x\in [p]:\neg\varphi(x,a)\}
\end{equation} 
is countable: If $A$ is uncountable, then by the perfect set theorem there
is a condition $q \leq p$ with $[q]\subseteq A$, and $q\forces_{\bS} \neg\varphi(\dot x_G,\check a)$ by $\Pi^1_2$-absoluteness;
if, on the other hand, $A$ is countable, then $1_{\bS} \forces \dot x_G \notin \check A$, and by $\Pi^1_1$-absoluteness $p\forces (\forall x \in [\check p]\setminus \check A)\; \varphi(x,\check a)$, proving the equivalence.

If $A$ is countable, (the proof of) the effective perfect set theorem \cite[Theorem~4F.1,~p.~243]{moschovakis2009descriptive} gives 
a sequence $(x_n)_{n\in\omega}\in L_{\omega_1^a}[a]$ such that
$$
\{x\in 2^{\omega}:\neg\varphi(x,a)\}=\{x_n:n\in\omega\}.
$$
Thus $p\forces_{\bS} \varphi(\dot x_G,\check a)$ if and only if
\[
(\exists (x_n)_{n\in\omega}\in\Delta^1_1[a])(\forall x \in 2^\omega) \big [ \neg\varphi(x,a)\implies (\exists n)\; x=x_n \big],
\]
which is $\Pi^1_1$ by \cite[Corollary~4.19,~p.~53]{mansfield1985recursive}.

For $\bM$, $p\forces_{\bM} \varphi(\dot x_G,\check a)$ precisely if $\{x\in 2^\omega:\neg\varphi(x,a)\}$ is contained in a $K_\sigma$ set by \cite{kechris1977notion} (or see \cite[Corollary~21.23,~p.~178]{kechris1995}). The rest of the proof is analogous to the above.
\end{proof}

Secondly, we will use the following well-known fact to give us a practical way of talking about names for reals. For the sake of completeness, we include a proof. For the rest of this section, let $\baseset = \omega$ when $\bP=\bM$, and $\baseset=2$ when $\bP=\bS$.

\begin{fact}\label{f.names}
Sacks and Miller forcing forcing have \emph{continuous reading of names for reals}: If $p\in \bP$, $\dot x$ is an $\bP$-name and $p\forces \dot x\in\omega^\omega$, then there is $\eta:\baseset^\omega\to\omega^\omega$ continuous and $q\leq p$ such that $q\forces \eta(\dot x_G)=\dot x$.
\end{fact}

Any continuous function $\eta\colon \baseset^{\omega} \rightarrow \omega^\omega$ arises from a monotone map between trees $\varphi\colon \baseset^{<\omega} \rightarrow \omega^{<\omega}$; in the notation of \cite[Definition~2.5,~p.~7]{kechris1995}, $\eta = \varphi^*$. So we can regard the countable object $\varphi$ as a `code' for $\eta$. To say that $\varphi$ gives rise to a total function is $\Pi^1_1$, whence absolute. We adopt the convention that 
if $\eta$ is coded by $\varphi$ in this sense, i.e., $\eta = \varphi^*$, then after moving to a forcing extension we still use $\eta$ to refer to the function coded by $\varphi$ there, i.e., $\varphi^*$ \emph{as interpreted in the current model}.
Without this convention the statement itself of Fact \ref{f.names} makes little sense.

For the proof of Fact~\ref{f.names}, we use the following terminology. 
Momentarily fix a tree $T \subseteq \baseset^{<\omega}$. 
Given $s \in \baseset^{<\omega}$, define
$(T)_s=\{ t \in T : s \subseteq t \}$. 
We call $t$ an \emph{immediate successor of $s$ (in $T$)} if $s \subseteq t$, $t\in T$, and $\lh(s) = \lh(t) -1$. 
Call $t \in T$ a \emph{full splitting node (in $T$)}  
if the set 
of  immediate successors of $t$ in $T$ has size $\baseset$; and for $n \in \omega\setminus\{0\}$ call $t$ an \emph{$n$-th full splitting node (in $T$)} if 
in addition  $\{s \in T \colon s \subseteq t\}$ contains exactly $n$ full splitting nodes.
Finally, we write $(T)^*_{\leq n}$ for the set of immediate successors of $n$-th full splitting nodes.

\begin{proof}[Proof of Fact \ref{f.names}.]
For $n\in\omega$, the set $D_n$ of $p\in\bP$ which decide a value for $\dot x (\check n)$ is dense and open.
We construct $q$ in a typical fusion argument: 
Let $p_0=p$ and inductively find $p_{n+1} \leq p_n$ such that 
 $(p_{n+1})^*_{\leq n+1}  = (p_n)^*_{\leq n+1}$ and whenever $s \in (p_{n+1})^*_{\leq n+1}$, then
 $(p_{n+1})_s \in D_n$.
Let $q$ be the greatest lower bound in $\bP$ of the sequence $p_0, p_1, \hdots$ (it exists as for each $k$, $(p_n)^*_{\leq k+1}$ is eventually constant in $n$).

For each $n \in \omega$ there are sequences $\{ s^n_0, s^n_1, \hdots \} \subseteq \baseset^{<\omega}$ and 
$\{ k^n_0, k^n_1, \hdots \}\subseteq \omega$ such that 
$\{ s^n_0, s^n_1, \hdots \} = (q)^*_{n+1}$ and
\[
q\forces \check{s}^n_i \subseteq \dot x_G \implies \dot x(\check n)=\check{k}^n_i.
\]
For $r\in[q]$, defining $\eta( r)( n) = k^n_i$ whenever $s^n_i \subseteq r$,
we have that  $\eta\colon[q]\rightarrow \omega^\omega$ is continuous and $q \forces \eta(\dot x_G)=\dot x$.
We can easily extend $\eta$ to a continuous function defined on all of $\baseset^\omega$.
\end{proof}

Lastly, we need a Ramsey-theoretic statement, Corollary \ref{c.Galvin} below, which in the Sacks case follows from the following theorem due to Galvin (see \cite[Theorem~19.7,~p.~145]{kechris1995}).
\begin{theorem}[Galvin's Ramsey theorem for Polish spaces]\label{t.Galvin}
Let $Y$ be a perfect Polish space, and suppose $$
[Y]^2=P_0\cup P_1
$$ 
is a partition of $[Y]^2$ into Baire measurable pieces. Then there is $C\subseteq Y$ perfect and $i\in\{0,1\}$ such that $[C]^2\subseteq P_i$.
\end{theorem}
\noindent

Given a binary relation $\mathcal R$, a set $A$ is called \emph{$\mathcal R$-complete} iff 
\[
(\forall x,y\in A)\ x\neq y\implies x\mathbin{\mathcal R} y.
\]
\begin{corollary}\label{c.Galvin}
Let $\mathcal R$ be an analytic relation on a standard Borel space $X$, let $\eta:\baseset^\omega\to X$ be a Borel function and let $p\in \bP$. Then there is  $q \in \bP$ such that $q \leq p$ and either $\eta([q])$ is $\mathcal R$-complete, or $\eta([q])$ is $\mathcal R$-discrete.
\end{corollary}

\begin{proof}
For $\bP = \bS$, assume $\mathcal R$ is symmetric and let $Y = [p]$ and
$$
P_0=\{\{x,y\}\in [Y]^2: \eta(x) \mathbin{\mathcal R}\eta(y)\},
$$
$P_1=[Y]^2\setminus P_0$. Let $C$ be given by Galvin's theorem and pick $q \in \bP$ such that $q \leq p$ and $[q]=C$. \
If $[C]^2\subseteq P_1$ then $\eta(C)$ is $\mathcal R$-discrete and if $[C]^2\subseteq P_0$ then $\eta(C)$ is $\mathcal R$-complete.

For $\bP=\bM$, as was pointed out by Stevo Todor\v{c}evi\'c, we may use  \cite[Corollary~5.68,~p.~121]{todorcevic2010} instead of Theorem \ref{t.Galvin} to derive the result.
\end{proof}

\begin{remark}
In the previous corollary, it is perfectly acceptable that $\eta$ is constant, say. In that case $\eta(2^\omega)$ is both $\mathcal R$-complete and $\mathcal R$-discrete at the same time.
\end{remark}

\begin{definition}\label{d.galvinwitness}
For $\mathcal R$ and $\eta$ as in Corollary \ref{c.Galvin}, call $q\in\bP$ a \emph{Galvin witness} for $\eta$ (and $\mathcal R$) if $\eta([q])$ is either $\mathcal R$-complete or $\eta([q])$ is $\mathcal R$-discrete.
\end{definition}

Note that $\eta([q])$ being $\mathcal R$-discrete is $\Pi^1_1[a]$ uniformly in $\eta$ and $q$, and $\eta([q])$ being $\mathcal R$-complete is $\Pi^1_2[a]$. 
In particular both are absolute for class models by Levy-Shoenfield, and thus so is the property of being a Galvin witness.

\medskip
Now we are ready to prove the main theorem of this section:
\begin{proof}[Proof of Theorem \ref{t.sacks}]
Let $\mathcal R \subseteq (\omega^\omega)^2$ be $\Sigma^1 _1[a]$;
the proof relativizes easily to the parameter $a$, so we suppress it below.

It suffices to produce a $\Sigma^1_2$ formula $\varphi$ which defines a maximal $\mathcal R$-discrete set in any $\bP$-generic extension of $L$, since if $\mathcal{A}$ is $\Sigma^1_2$ and maximal $\mathcal{R}$-discrete set, then $\mathcal{A}$ is in fact $\Delta^1_2$, since
$$  
x \not\in \mathcal{A} \iff ( \exists y \in \mathcal{A})\ x\mathbin{\mathcal{R}}y \wedge x\neq y.
$$

Below we identify $\bP \times C(\baseset^\omega,\omega^\omega) $ with a $\Pi^1_1$ subset of $\omega^\omega$, by identifying both $C(\baseset^\omega,\omega^\omega)$ and $\bP$ with subsets of $\omega^\omega$ (see the remark after Fact \ref{f.names}) and identifying $\omega^\omega$ and $(\omega^\omega)^2$ via some fixed effective bijection\footnote{For the case $\baseset=2$, we could alternatively use that the set of continuous functions $C(2^\omega, \omega^\omega)$ has an effective presentation as a Polish metric space, and so we can regard it as a $\Pi^0_2$ subset of $2^\omega$.}.

Working in $L$, fix an enumeration $\langle (p_\xi,\eta_\xi):\xi<\omega_1\rangle$ of $\bP\times C(\baseset^\omega,\omega^\omega)$ such that
$$
\xi < \delta \implies (p_\xi,\eta_\xi) <_L (p_{\delta},\eta_{\delta}).
$$
By recursion on $\omega_1$, we will define a sequence $\langle q_\xi:\xi<\omega_1\rangle$, such that the following are satisfied:
\begin{enumerate}[label=(\roman*),ref=\roman*]
\item\label{def.clause.zero}\label{def.clause.first} $q_\xi$ is a (not necessarily perfect) subtree of $p_\xi$.
\item \label{def.clause.one} If $p_\xi$ does not force
\begin{equation}\label{def.clause.equ.p.forces}
(\forall \delta<\xi)(\forall y\in [q_\delta])\ \eta_\xi(\dot x_G) \mathbin{\not\!\!\mathcal{R}} \eta_\delta(y)
\end{equation}
then $q_\xi = \emptyset$.
\item \label{def.clause.two} If $p_\xi$ does force \eqref{def.clause.equ.p.forces} and $q \in \bS$ is $\leq_L$-least such that $q\leq p_\xi$ and
\begin{equation}\label{eq.q}
(\forall x \in [q])(\forall \delta<\xi)(\forall y\in [q_\delta])\ \eta_\xi(x) \mathbin{\not\!\!\mathcal{R}} \eta_\delta(y),
\end{equation}
then 
\begin{enumerate}[label=(\ref{def.clause.two}.\alph*), ref=\ref{def.clause.two}.\alph*, align=left]
\item \label{def.subclause.a}  if $\eta_\xi([q])$ is $\mathcal R$-discrete, $q_\xi = q$;
\item \label{def.subclause.b} if $\eta_\xi([q])$ is \emph{not} $\mathcal R$-discrete, $q_\xi \subseteq q$ is the pruned subtree whose unique branch is the left-most branch of $q$.
\end{enumerate}
\item\label{def.clause.three}\label{def.clause.last} The set $\mathcal{A}^0=\{ (q_\xi, \eta_\xi) : \xi <\omega_1 \}$ is $\Sigma^1_2$.
\end{enumerate}
Above it is implicit in \eqref{def.clause.two} that $q$ exists; this follows since in case \eqref{def.clause.two}  the analytic set
$$
\{ x \in [p_\xi] \colon (\exists \delta<\xi)(\exists y\in [q_\delta])   \; \eta_\xi(x) \mathbin{\mathcal{R}} \eta_\delta(y)\}
$$
must be countable since $p_\xi$ forces \eqref{def.clause.equ.p.forces} and by the proof of Fact~\ref{f.complexityforces}.

Suppose for now that $\langle q_\xi:\xi<\omega_1\rangle$ satisfies \eqref{def.clause.first}--\eqref{def.clause.last} above. Then let $\varphi(y)$ be the $\Sigma^1_2$ formula
$$
( \exists q,\eta) (\exists x \in [q]) (q,\eta)  \in \mathcal{A}^0 \wedge  y = \eta(x).
$$
Clause \eqref{def.clause.two} ensures that $\varphi$ defines an $\mathcal R$-discrete set in any model. For maximality, suppose, seeking a contradiction, that
$$
p\forces (\exists x\in\omega^\omega)(\neg\varphi(x)\wedge (\forall y)(\varphi(y)\implies x\not\!\!\mathcal R\ y)).
$$
By Fact~\ref{f.names} we may assume (replacing $p$ with a stronger condition if necessary) that there is a total continuous function $\eta:\baseset^\omega\to\omega^\omega$ such that
\begin{equation}\label{eq.forces}%
p\forces \neg\varphi(\eta(\dot x_G))\wedge (\forall y \in\Omega^\omega)(\varphi(y)\implies \eta(\dot x_G)\not\!\!\mathcal R\ y).
\end{equation}
By Corollary \ref{c.Galvin}, we may assume that $p$ is a Galvin witness for $\eta$, since otherwise, we can again replace $p$ by a stronger condition. 

Let $\xi< \omega_1$ be such that $(p,\eta)=(p_\xi,\eta_\xi)$. Then $p_\xi$ forces \eqref{def.clause.equ.p.forces} and so clause \eqref{def.clause.two} applies. Let $q\leq p_\xi$ be as in clause \eqref{def.clause.two}.

If $\eta([q])$ is $\mathcal R$-discrete then $q_\xi=q$. Since $q\forces \dot x_G\in [q]$ it follows that $q\forces \varphi(\eta(\dot x_G))$, contradicting that $p\forces \neg\varphi(\eta(\dot x_G))$.

So it must be that $\eta([q])$ is not $\mathcal R$-discrete, and so clause \eqref{def.subclause.b} applies. Since $p$ is a Galvin witness for $\eta$, it follows that $\eta([p])$ is $\mathcal R$-complete. Let $z\in [q_\xi]$ be the unique branch through $q_\xi$. Then
$$
(\forall x\in [p])\ \eta(x)=\eta(z)\vee \eta(x)\ \!\mathcal R\ \!\eta(z),
$$
and since this is $\Pi^1_2[z]$, it follows by Shoenfield's absoluteness theorem that $p\forces \eta(\dot x_G)=\eta(\check z)\vee\eta(\dot x_G)\ \!\mathcal R\ \!\eta(\check z)$, contradicting (\ref{eq.forces}).

\medskip 

It is routine that $\langle q_\xi:\xi<\omega_1\rangle$ satisfying \eqref{def.clause.zero}--\eqref{def.clause.two} above can be found. In fact \eqref{def.clause.one} and \eqref{def.clause.two} determine the sequence $\langle q_\xi:\xi<\omega_1\rangle$ uniquely. So proving the following claim will finish the proof.

\begin{claim}
The set $\mathcal{A}^0=\{ (q_\xi,\eta_\xi) : \xi <\omega_1\}$ is $\Sigma^1_2$.
\end{claim}
{\it Proof of claim}:
The set of sequences $\vec{x}=\langle (p^*_n,\eta^*_n): n \in \alpha\rangle$, for $\alpha \leq \omega$,  such that 
\begin{equation}\label{init}
(\exists \xi <\omega_1)\ \{(p^*_n,\eta^*_n): n \in \alpha\} = \{ (p_\delta,\eta_\delta):\delta < \xi \}.
\end{equation}
is $\Sigma^1_2$; since $\leq_L$ is a strongly $\Delta^1_2$ well-ordering of $\omega^\omega$ (see \cite{harrington1975} or \cite{t-weiss}) and $\bP\times C(\baseset^\omega,\omega^\omega)$ was identified with a $\Pi^1_1$ subset of $\omega^\omega$, this follows easily by the proof of \cite[Exercise~5A.1,~p.~287]{moschovakis2009descriptive}. 
Let $\Psi(\vec{x})$ be a $\Sigma^1_2$ formula equivalent to \eqref{init}.

Now assume $\Psi(\vec{x})$ holds and suppose $\vec{y} = \langle q^*_n : n\in\alpha\rangle$ 
enumerates the sequence $\{ q_\delta:\delta<\xi\}$ in such a way that for all $\delta<\xi$ and $n \in \alpha$ we have $p^*_n = p_\delta \iff q^*_n=q_\delta$.

Observe that \eqref{def.clause.one} and \eqref{def.clause.two} are $\Sigma^1_2$  uniformly in $(\vec{x},\vec{y},p_\xi,\eta_\xi,q_\xi)$: \eqref{def.clause.one} is $\Pi^1_1$ in these parameters %
by \ref{f.complexityforces}.
The property expressed in \eqref{eq.q} is $\Pi^1_1$ in $(\vec{x},\vec{y},\eta_\xi,q)$. Thus, that $q$ be minimal with that property is $\Sigma^1_2$ in these parameters, 
as $\leq_L$ is strongly $\Delta^1_2$ (by `closure under bounded quantification').
Clauses
\eqref{def.subclause.a} and  \eqref{def.subclause.b} are Boolean combinations of $\Sigma^1_1$ formulas in the parameters $(\eta_\xi,q_\xi,q)$ by the remarks following Definition \ref{d.galvinwitness}.
Thus, \eqref{def.clause.two} is easily seen to be $\Sigma^1_2$ in $(\vec{x},\vec{y},p_\xi,\eta_\xi,q_\xi)$.

So we may express the conjuction of \eqref{def.clause.one} and \eqref{def.clause.two} by a  $\Sigma^1_2$ formula 
\[
\Theta(q, p_\xi, \eta_\xi, \vec{x}, \vec{y}),
\]
i.e., for $\vec{x}$ and $\vec{y}$ as above, $\Theta(q, p_\xi, \eta_\xi, \vec{x}, \vec{y})$ holds if and only if $q = q_\xi$.

Thus $\vec y=\langle q^*_n  : n \in \alpha \rangle$, with $\alpha \leq \omega$, enumerates an initial segment of $\langle q_\xi : \xi <\omega_1 \rangle$ exactly if the following formula holds:
\begin{equation}\label{formula}
(\exists \vec{ x} )\;  \Psi(\vec{x}) \wedge \dom(\vec{x})=\dom(\vec{y}) \wedge [ (\forall n < \dom(\vec{x})) \; \Theta(q^*_n, p^*_n, \eta^*_n, \vec{x} \restriction n, \vec{y}\restriction n) ]
\end{equation}
where we abuse notation by writing
$\vec{x}\restriction n$ for 
\[
\langle (p^*_m, \eta^*_m ) : m < \dom(\vec{x}), (p^*_m, \eta^*_m ) \leq_L (p^*_n, \eta^*_n )  \rangle
\]
 (similarly for $\vec{y}\restriction n$). 
The formula \eqref{formula} is easily seen to be equivalent to a $\Sigma^1_2$ formula; it follows that $\mathcal A^0$ is $\Sigma^1_2$.
\hfill {\tiny Claim.}$\dashv$

\end{proof}
We also get the following effective corollary for $\mathbf{\Sigma}^1_1$ relations:
\begin{corollary}
Under the hypothesis of Theorem \ref{t.sacks}, if $a' \in L[x]$ and $\mathcal R$ is a $\Sigma^1_1[a']$ binary relation on a effectively presented Polish space $X$, then in $L[x]$ there is a $\Delta^1_2[a']$ maximal $\mathcal R$-discrete set in $X$.
\end{corollary}
\begin{proof}
We may assume $X=\omega^\omega$. 
If $a'\in L$, then Theorem \ref{t.sacks} gives a $\Delta^1_2[a']$ formula $\varphi$ defining a  maximal $\mathcal R$-discrete set in $L[x]$. 
If $a'\notin L$ then $L[x]=L[a']$ (as $x$ is of minimal degree, see \cite{halbeisen2012combinatorial}), and  Theorem \ref{t.sacks} provides a $\Delta^1_2[a']$ formula $\varphi$ defining a  maximal $\mathcal R$-discrete set in $L[a']$---and, incidentally, its Sacks (resp.\ Miller) extension---starting from the strongly $\Delta^1_2[a']$ well-ordering of $L[a']$.
\end{proof}

\medskip

It's simple to axiomatize a class of forcings for which the above proof goes through.
A forcing $\bP$ is called \emph{strongly arboreal} \cite{ikegami2010} if and only if its conditions are perfect trees on $\baseset$ where $\baseset\subseteq\omega$, ordered by inclusion, and for any $p \in \bP$ and $s \in p$ also $(p)_s \in \bP$.
Any extension of $V$ by a $(\bP,V)$-generic filter $G$ is generated by the single `generic' real $\bigcap_{p\in G} [p]$; its name we denote by $\dot x^{\bP}_G$. 
A real is called $(\bP,V)$-generic over if and only if it arises in this way.
For example, Sacks, Miller, Mathias and Laver are (equivalent to) strongly arboreal forcings (see e.g.\ \cite{brendle2005silver} and \cite{halbeisen2012combinatorial}).  
\begin{theorem}\label{t.arboreal}
Let $\bP$ be a strongly arboreal forcing such that:
\begin{enumerate}[label=(\Alph*),align=left]
\item \label{P.names} %
$\bP$ has Borel reading of names (in the sense of \cite[Proposition~2.3.1,~p.~29]{zapletal2008forcing}).
\item \label{P.forces} If $\varphi(x,y)$ is a $\Pi^1_1$ predicate then
$
\{(p,a)\in\bP\times\omega^\omega:p\forces \varphi(\dot x^{\bP}_G,\check a)\}
$
is $\Delta^1_2$.
\item\label{galvin.gen} The analogue of Galvin's theorem holds for $\bP$: For $\mathcal R$ as in Theorem \ref{t.sacks}, a Borel function $\eta\colon \baseset^\omega\rightarrow \omega^\omega$ and $p\in \bP$, there is $q\in \bP$, $q \leq p$ such that $q$ is a Galvin witness for $\mathcal R$ and $\eta$.
\end{enumerate}
Then the analogue of Theorem \ref{t.sacks} holds when $x$ is a $(\bP,L[a])$-generic real.
\end{theorem}
We mention without proof that \ref{P.forces} can be replaced by: for all countable transitive $M$ and $p\in\bP\cap M$ there is $q\leq p$ s.t.\ 
any $r \in [q]$ is a $(\bP,M)$-generic.
\begin{proof}[Proof of Theorem \ref{t.arboreal}]
One difference to the proof of Theorem \ref{t.sacks} is how we obtain the enumeration $\langle p_\xi, \eta_\xi : \xi<\omega_1\rangle$ at the beginning. 
The second coordinate now has to enumerate all codes for total Borel functions; the set of such codes is $\Pi^1_1$ (to see this, observe that if $f$ is $\Delta^1_1[a]$ then $f(x)$ is $\Delta^1_1[x,a]$; now use \cite[Corollary~4.19,~p.~53]{mansfield1985recursive}), so the proof goes through as before. 
It remains to notice that when even just requiring \ref{P.forces}, clauses \eqref{def.clause.one} and \eqref{def.clause.two} remain $\Sigma^1_2$.
\end{proof}

\section{A co-analytic mof in the Miller and Sacks extensions}\label{s.mof}

\begin{theorem}\label{t.pi11mof}
If $x$ is a Miller or Sacks real over $L[a]$, then
$$
L[a][x]\models \text{\emph{``there is a $\Pi^1_1[a]$ mof in $P(2^\omega)$''}}.
$$
\end{theorem}
This follows immediately from Theorem \ref{t.discrete} together with the following lemma.
 
\begin{lemma}\label{l.code}
If there is a $\Sigma^1_2[a]$ mof in $P(2^\omega)$, then there is a $\Pi^1_1[a]$ mof.
\end{lemma}
\begin{proof}
We suppress the parameter $a$ below. 

The proof is based on a slight simplification of the coding method from \cite{Fischer2010}. Let $P_c(2^\omega)$ denote the set of \emph{atomless} Borel probability measures on $2^\omega$, i.e., $\mu \in P(2^\omega)$ such that $\mu(\{x\})=0$ for any $x$. This set is $\Pi^0_2$ as a subset of $P(2^\omega)$, see \cite[Lemma 2.1]{Fischer2010}.
Given $\mu \in P_c(2^\omega)$, 
let $y$ be the left-most branch of the tree
$$
\{t\in 2^{<\omega}:\mu(N_t) > 0\},
$$
where $N_s=\{ x \in 2^\omega : s \subseteq x \}$ is the basic open neighborhood determined by $s \in \omega^n$.
Let $n(0), n(1), \hdots$ enumerate the infinite set of $n$ such that $\mu(N_{y\upharpoonright n\mathbin{{}^\frown} 0})>0$ and $\mu(N_{y\upharpoonright n\mathbin{{}^\frown} 1})>0$ (here we use that $\mu$ is atomless) and define $G(\mu) \in 2^\omega$ by
\[
 G(\mu)(i)=\begin{cases}
             0 &\text{if }\mu (N_{y\upharpoonright n(i)\mathbin{{}^\frown} 0}) \geq \mu (N_{y\upharpoonright n(i)\mathbin{{}^\frown} 1}) \\
            1 &\text{otherwise.}
           \end{cases}
\]

We say $G(\mu)$ is ``coded'' by $\mu$. As in \cite{Fischer2010}, we can find a $\Delta^1_1$ coding function $F\colon P_c(2^\omega)\times 2^\omega\to P_c(2^\omega)$ such that for all $\mu \in P_c(2^\omega)$ and $y\in 2^\omega$, $F(\mu, y)$ is absolutely equivalent to $\mu$ and codes $y$, that is, \
$G(F(\mu,y))=y$.

\medskip

Now let $\mathcal{A}$ be a $\Sigma^1_2$ mof in $P(2^\omega)$. 
By subtracting from each measure in $\mathcal A$ a measure concentrating on a countable set and re-normalizing the remaining non-zero measures, we may assume $\mathcal A\subseteq P_c(2^\omega)$ (here we use the ccc-below property) and that $\mathcal A$ is maximal among orthogonal families of atomless measures. 
Let $R$ be $\Pi^1_1$ such that $\mu \in \mathcal{A} \iff (\exists y) \;  R(\mu, y)$. By $\Pi^1_1$ uniformization, we can assume $R$ is a functional relation, i.e.
 \[
 (\forall x \in \dom( R)) \; 
   (\exists! y)\; R(x,y).
\]
   Fix a $\Sigma^0_1$ bijection $(x,y)\mapsto x \oplus y$ from $(2^\omega)^2$ to $2^\omega$,
   and let $x \mapsto (x)_i$, for $i \in \{0,1\}$ be the pair of maps such that
   for all $z\in2^\omega$, $z = (z)_0 \oplus (z)_1$ (i.e., the components of the inverse of our bijection). Let $g:2^\omega\to P_c(2^\omega)$ be a $\Delta^1_1$ bijection.

Define $\mathcal A'\subseteq P_c(2^\omega)$ by letting $\mu' \in \mathcal{A}'$ just in case $\mu' \in P_c(2^\omega)$ and
\begin{multline}\label{theothermof}
(\forall z, \mu, y) \; \big[ z=G(\mu')  \wedge \mu=g((z)_0) \wedge y=(z)_1 \big] \implies\\
 \big[ \mu'= F(\mu, z) \wedge R(\mu, y)  \big].
\end{multline}
Clearly $\mathcal A'$ is $\Pi^1_1$ and moreover $\mathcal A'$ is an orthogonal family of atomless measures which is maximal among the atomless measures
since
$$\mathcal A' = \{F(\mu, g^{-1}(\mu) \oplus y) : \mu\in \mathcal A,  y \in 2^\omega, R(\mu, y) \}.$$

By enlarging $\mathcal{A}'$ to contain all Dirac measures (i.e., probability measures concentrating on a single point), we obtain a $\Pi^1_1$ mof.
\end{proof}

\section{No co-analytic mofs in the Mathias extension}\label{sec.mathias}\label{s.mathias}

The purpose of this section is to complement Theorem \ref{t.pi11mof} by showing that its conclusion fails when $r$ is a Mathias real over $L$.

\begin{theorem}\label{t.mathiasmof}
If all $\Sigma^1_2[a]$ (equivalently, all $\Delta^1_2[a]$) sets of reals are  
Ramsey  there is no $\Sigma^1_2[a]$ mof.
In particular, if $r$ is a Mathias real over $L[a]$, then
$$
L[a][r]\models \text{\emph{``there is {\bf no} $\Sigma^1_2[a]$ mof''}}.
$$
\end{theorem}

In the process of proving Theorem~\ref{t.mathiasmof} we will also obtain a new proof that there are no analytic mofs, see the end of this section.

\medskip

The proof of Theorem \ref{t.mathiasmof} requires multiple steps. We start by defining a way of assigning to each $x\in [\omega]^\omega$ a product measure on $2^{\omega}$. Let $f_x:\omega\to\omega$ denote the unique increasing function such that $x=f_x[\omega]$. Then define a sequence $\alpha^x\in [\frac 1 4,\frac 3 4]^\omega$ by
$$
\alpha^x_n=\begin{cases}
\frac 1 4+\frac 1 {2\sqrt{f^{-1}_x(n)+1}} & n\in x,\\
\frac 1 4 & n\notin x,
\end{cases}
$$
and define $\mu^x\in P(2^\omega)$ by
$$
\mu^x=\prod_{n\in\omega} (\alpha^x_n\delta_0+(1-\alpha^x_n)\delta_1),
$$
where $\delta_i\in P(\{0,1\})$ is the Dirac measure concentrated at $i\in\{0,1\}$. For $x,y\in [\omega]^\omega$, let
$$
\rho(x,y)=\sum_{n\in\omega}(\alpha^x_n-\alpha^y_n)^2.
$$
Note that if $z\subseteq y\subseteq x$ then $\rho(y,x)\leq\rho(z,x)$.

The intention behind the definition of $\alpha^x$ is to be able to use Kakutani's theorem on equivalence and orthogonality of product measures. Specifically, \cite[Corollary~1,~p.~222]{kakutani1948} gives
\begin{equation}\label{eq.kakutani}
\mu^x\sim \mu^y\iff \rho(x,y)<\infty
\end{equation}
and
$$
\mu^x\perp \mu^y\iff \rho(x,y)=\infty.
$$
Let then
$$
\mathcal F=\{g\in [\omega]^\omega: \rho(g,\omega)<\infty\},
$$
and define a binary operation $\cdot$ on $[\omega]^\omega$ by
$$
x\cdot y=f_y\circ f_x[\omega].
$$
We can think of the operation $x\cdot y$ as follows: $f_y$ identifies $\omega$ and $y$, and $x\cdot y$ is the copy of $x$ inside of $y$ under this identification.

\begin{prop}\label{p.monoid}\ 

$(i)$ For all $x,y,z\in [\omega]^\omega$
$$
\sqrt{\rho(x,y)}\leq\sqrt{\rho(x,z)}+\sqrt{\rho(z,y)}
$$
holds, even if $\rho(x,y)$ is infinite. In particular, $\sqrt\rho$ is finite on $\mathcal F$ and defines a complete metric on $\mathcal F$, inducing a Polish topology on $\mathcal F$. In this topology, $\mathcal F$ is a perfect Polish space.

$(ii)$ The operation $\cdot$ is associative and makes $[\omega]^\omega$ a monoid with the unit being $\omega$.

$(iii)$ For all $g\in\mathcal F$ and $x\in [\omega]^{\omega}$ we have $\rho(g\cdot x,x)=\rho(g,\omega)$ and $\mu^{g\cdot x}\sim\mu^x$. It follows that $\mathcal F$ is closed under the operation $\cdot$, and so is a monoid with unit $\omega$.

$(iv)$ If $|x\triangle y|<\infty$ then $\rho(x,y)<\infty$.

$(v)$ For any $g\in\mathcal F$ and $k\in\omega$, it holds that $g\cdot(\omega\setminus\{k\})\in\mathcal F$, and
$$
\lim_{k\to\infty}\rho(g\cdot(\omega\setminus \{k\}),g)=0.
$$

$(vi)$ For any $x\in [\omega]^\omega$, the equivalence relation $\sim^x$, defined in $[x]^\omega$ by $z\sim^x y\iff \mu^z\sim\mu^y$, has meagre, dense classes in the Polish topology on $[x]^\omega$.
\end{prop}
\begin{proof}
$(i)$ Note that $\sqrt{\rho(x,y)}=\|\alpha^x-\alpha^y\|_2$, where $\|\cdot\|_2$ is the 2-norm. The norm inequality
$$
\|\alpha^x-\alpha^y\|_2\leq \|\alpha^x-\alpha^z\|_2+\|\alpha^z-\alpha^y\|_2
$$
holds in the strong sense that if the left hand side is infinite, then so is the right hand side (use that $\ell_2(\omega)$ is closed under addition). This establishes the inequality in $(i)$. The map $g\mapsto \alpha^g-\alpha^\omega$ is then an isometric embedding of $\mathcal F$ into the Hilbert space $\ell_2(\omega)$. It is straight-forward to check that the image under $g\mapsto\alpha^g-\alpha^\omega$ is closed in $\ell_2(\omega)$. Finally, $(v)$ implies that $\mathcal F$ has no isolated points.

$(ii)$ follows immediately from the definitions. $(iii)$ is follows easily from $(i)$, eq. (\ref{eq.kakutani}), and the definition of $\rho$. For $(iv)$ and $(v)$ we need:

\medskip

{\bf Claim:} $\omega\setminus k\in\mathcal F$ for any $k\in\omega$.

\medskip

{\it Proof of Claim}:
$$
\sum_{n\geq k} (\alpha^\omega_n-\alpha^{\omega\setminus k}_n)^2=\sum_{n\geq k}\left(\frac 1 {2\sqrt {n+1}}-\frac 1 {2\sqrt {n+1-k}}\right)^2,
$$
and the right hand side above converges since, using a small amount of calculus, we have $(n-k)^{-\frac 1 2}-n^{-\frac 1 2}\leq n^{-\frac 3 2}$ for $n>k$ sufficiently large. \hfill {\tiny Claim.}$\dashv$

\medskip

$(iv)$ Suppose $|x\triangle y|<\infty$ and let $z=x\cup y$. Let $k$ be such that $z\setminus k\subseteq x\cap y$ and $k_0=|z\cap k|$. Then 
$$
\rho(z,x),\rho(z,y)\leq\rho(z,x\cap y)\leq \rho(z,z\setminus k)=\rho(z,(\omega\setminus k_0)\cdot z)<\infty,
$$
with the last inequality following from the previous claim and $(iii)$.

$(v)$ Since $\omega\setminus \{k\}\supseteq\omega\setminus (k+1)$ the claim gives $\omega\setminus \{k\}\in\mathcal F$ for all $k\in\omega$. Now
$$\label{eq.sums}
\rho(g\cdot (\omega\setminus\{k\}),g)=\sum_{n\geq k}(\alpha^{g\cdot (\omega\setminus\{k\})}_n-\alpha_n^g)^2=\sum_{n\geq k} (\alpha^{g\cdot (\omega\setminus\{0\})}_n-\alpha_n^g)^2,
$$
and the last sum tends to $0$ as $k\to\infty$ since $\rho(g\cdot(\omega\setminus\{0\}),g)<\infty$.

$(vi)$ By eq. (\ref{eq.kakutani}) the relation $\sim^x$ is $F_\sigma$. $(iv)$ immediately gives that all $\sim^x$ classes are dense in $[x]^\omega$. On the other hand, it is easy to check that if $y\in [x]^\omega$ and the complement of $f^{-1}_x(y)$ is not in the summable ideal, then $\rho(y,x)=\infty$. So $\sim^x$ has at least two (necessarily dense) classes. It follows that the complement of any $\sim^x$ class is a dense $G_\delta$ set in $[x]^\omega$, whence each $\sim^x$ class is meagre in $[x]^\omega$.
\end{proof}

\begin{remark}
$(v)$ in the previous proposition intends to say that multiplication on the right in the monoid $\mathcal F$ has at least \emph{some} amount of continuity at $\omega$ (the identity). By contrast, $(iii)$ shows that left multiplication in $\mathcal F$ is continuous at $\omega$. We do not know if right multiplication is actually continuous at the identity, but it seems unlikely.
\end{remark}

Let $X$ be a Polish space. Recall that the equivalence relation $F_2^X$ on $X^\omega$ is defined by
$$
\vec{x}\ \! F_2^X\ \! \vec{y}\iff \{\vec x_n:n\in\omega\}=\{\vec y_n:n\in\omega\}.
$$

\begin{lemma}\label{l.hjorth}
Let $\vartheta: [\omega]^\omega\to X^\omega$ be a continuous function (w.r.t.\ the Polish topologies), and suppose $\vartheta$ is $(\mathcal F,F_2^X)$-equivariant, that is,
$$
(\forall g\in\mathcal F)(\forall x\in [\omega]^\omega)\ \vartheta(g\cdot x)\sm F_2^X\sm\vartheta(x).
$$
Then there is a non-empty open set $U_0$ such that $\vartheta_0\upharpoonrighttrict U_0$ is constant, where, in general, $\vartheta_l$ is defined by $\vartheta_l(x)=\vartheta(x)_l$ for $l\in\omega$.
\end{lemma}

\begin{proof}[Proof \emph{(\`a la Hjorth).}] For the purpose of this proof, we identify $[\omega]^\omega$ with a $G_\delta$ subset of $2^\omega$ in the natural way. Define for each $l\in\omega$ a closed set
$$
A_l=\{(g,x)\in\mathcal F\times [\omega]^\omega:\vartheta_l(g\cdot x)=\vartheta_0(x)\},
$$
and let $l_0$ be least such that $A_{l_0}$ is non-meagre; this exists because $\mathcal F\times [\omega]^\omega=\bigcup_{l\in\omega} A_l$. Let $V\times U\subseteq A_{l_0}$ be open and non-empty, and fix $g_0\in V$. Using $(v)$ of Proposition \ref{p.monoid}, find $k_0$ such that
$$
(\forall k> k_0)\ g_0\cdot(\omega\setminus \{k\})\in V.
$$
Let $s_0\in 2^{<\omega}$ be such that $N_{s_0}\subseteq U$. 
By either making $s_0$ longer or $k_0$ larger, we may assume that $s_0$ is the characteristic function of a set with $k_0$ elements. 

\medskip

{\bf Claim:} If $y,z\in N_{s_0}$ differ on only one bit then $\vartheta_0(y)=\vartheta_0(z)$.

\medskip

{\it Proof of Claim}:
Suppose $n\in y$ and $n\notin z$. Let $k=f^{-1}_y(n)$, and note that $z=(\omega\setminus\{k\})\cdot y$ and $k>k_0$. Since $g_0\cdot(\omega\setminus\{k\})\in V$ we have
$$
\vartheta_0(z)=\vartheta_0(\omega\setminus\{k\}\cdot y)=\vartheta_{l_0}(g_0\cdot(\omega\setminus\{k\})\cdot y))=\vartheta_0(y).
$$
\hfill {\tiny Claim.}$\dashv$

\medskip

Now $U_0=N_{s_0}$ works, since the claim implies that the continuous function $\vartheta_0$ is constant on a dense set in $N_{s_0}$.
\end{proof}

\begin{proof}[Proof of Theorem \ref{t.mathiasmof}]
Assume that every $\Sigma^1_2$ set is completely Ramsey; in particular by \cite[0.9]{ihoda1989delta}, this holds in $L[r]$, where $r$ is a Mathias real over $L$. By Lemma \ref{l.code}, it is enough to show that there is no $\Pi^1_1$ mof. (For notational convenience, we suppress the parameter $a$.)

Suppose for a contradiction that $\mathcal A\subseteq P(2^\omega)$ is a $\Pi^1_1$ mof. Let $Q\subseteq [\omega]^\omega\times P(2^\omega)^\omega$ be
$$
Q=\{(x,(\nu_n)): (\forall n)(\nu_n\in\mathcal A\wedge\nu_n\not\perp\mu^x)\wedge (\forall \mu)(\mu\not\perp\mu^x\longrightarrow (\exists n)\nu_n\not\perp\mu)\}.
$$
(Thus $(x,(\nu_n))\in Q$ iff $(\nu_n)$ enumerates the countably many measures in $\mathcal A$ that are not orthogonal to $\mu^x$.) Since $\mathcal A$ is maximal the sections $Q_x$ are never empty, and by the $\Pi^1_1$ uniformization theorem, we can find a function $\vartheta:[\omega]^\omega\to P(2^\omega)^\omega$ which has a $\Pi^1_1$ graph, and such that $(x,\vartheta(x))\in Q$ for all $x\in [\omega]^\omega$. 
Note that if $\mu^x\sim\mu^y$ then 
$$
\{\vartheta_n(x):n\in\omega\}=\{\vartheta_n(y):n\in\omega\},
$$
so that by $(iii)$ of Proposition \ref{p.monoid}, $\vartheta$ is $(\mathcal F,F_2^{P(2^\omega)})$-equivariant.

It is easy to check that then $\vartheta^{-1}(U)$ is $\Sigma^1_2$ for every basic open set $U\subseteq P(2^\omega)^\omega$. Since every $\Sigma^1_2$ set is completely Ramsey, a standard argument \cite[Exercise~19.19,~p.~134]{kechris1995} shows that there is $x\in [\omega]^\omega$ such that $\vartheta\upharpoonrighttrict [x]^\omega$ is continuous (w.r.t.\ the Polish topology on $[x]^\omega$.)

From Lemma \ref{l.hjorth} it follows that there is a non-empty open set $U\subseteq [x]^\omega$ and $\nu\in P(2^\omega)$ such that $\vartheta_0(x)=\nu$ for all $x\in U$. By $(vi)$ of Proposition \ref{p.monoid} there is an uncountable (indeed a perfect) set $P\subseteq U$ such that if $x,y\in P$ and $x\neq y$, then $\mu^x\perp\mu^y$. Now for every $x\in P$ we have $\mu^x\not\perp\nu$, contradicting the ccc-below property of $\ll$.\end{proof}

The above line of argument also gives a new proof of the theorem of Preiss and Rataj, which we sketch below. Unlike the new proof that was given by Kechris and Sofronidis in \cite{kechris-sofronidis2001}, the proof below does not rely directly on Hjorth's turbulence theory. All the same, Lemma \ref{l.hjorth} above owes a debt to \cite[Lemma 3.14,~p.~42]{Hj:Book} that can scarcely be ignored.

\begin{theorem}[\cite{preiss-rataj1985}]
There are no analytic mofs in $P(2^\omega)$.
\end{theorem}

\begin{proof}[Sketch of proof] Suppose $\mathcal A$ were an analytic mof. By maximality, $\mathcal A$ would be Borel, and maximality along with the ccc-below property gives that the Borel set
$$
Q'=\{(x,\nu)\in [\omega]^\omega\times P(2^\omega): \mu^x\not\perp\nu\wedge\nu\in\mathcal A\}
$$
would have all vertical sections $Q'_x$ non-empty and countable. Then we could find countably many Borel functions $\vartheta_l:[\omega]^\omega\to P(2^\omega)$ such that
$$
Q'=\bigcup_{l\in\omega} \grph(\vartheta_l).
$$
The equivariance
$$
\mu^x\sim\mu^y\implies\{\vartheta_n(x):n\in\omega\}=\{\vartheta_n(y):n\in\omega\}
$$
is clear. Again, \cite[Exercise~19.19,~p.134]{kechris1995} would allow us to find $x\in [\omega]^\omega$ such that $\vartheta_l\upharpoonrighttrict [x]^\omega$ is continuous (w.r.t.\ the Polish topology) for all $l$. A contradiction is then obtained in exactly the same way it was in the proof of Theorem \ref{t.mathiasmof}.
\end{proof}

\section{Open problems}\label{s.questions}

Given the results of this paper, we pose the following questions:

\begin{enumerate}
\item Does the analogue of Theorem \ref{t.sacks} hold for the Laver extension? (Note that the analogue of Galvin's theorem is false for Laver forcing.)
\item Is there a model with a $\Pi^1_1$ mof such that in addition, for any $r\in 2^\omega$, there is a Miller real over $L[r]$? (For Sacks forcing, this question has been answered in \cite{schr}.)
\item Is the existence of a $\Pi^1_1$ mof consistent with $2^{\omega}=\omega_3$?
(For $2^{\omega}=\omega_2$, this question has also been answered in \cite{schr}).
\item Does Theorem \ref{t.sacks} fail for $\mathcal{R}$ which are $\Pi^1_1$? What about Theorem~\ref{t.sacks.ch}?
\item Are there natural forcing notions other than Sacks and Miller to which the hypothesis of Theorem \ref{t.arboreal} applies?
\end{enumerate}

\bibliographystyle{amsplain}
\bibliography{mof-alt}

\end{document}